\newif\ifpictures
\picturestrue

\newif\ifcomment
\commenttrue

\documentclass[12pt]{amsart}
\usepackage{latex_base}
\usepackage{macros}

\author{Khazhgali Kozhasov}

\address{Khazhgali Kozhasov, Technische Universit\"at Braunschweig, Institut f\"ur Analysis und Algebra, Universit\"atsplatz 2, 38106 Braunschweig,
 Germany\medskip}

\email{k.kozhasov@tu-braunschweig.de}



\title{On minimality of determinantal varieties}

\begin{document}

\maketitle

\begin{changemargin}{1.5cm}{1.5cm}
{\bf\noindent Abstract.}
We prove that semialgebraic sets of rectangular matrices of a fixed rank, of skew-symmetric matrices of a fixed rank and of real symmetric matrices whose eigenvalues have prescribed multiplicities are minimal submanifolds of the space of real matrices of a given size. 
\end{changemargin}
\vspace{1cm}

\section*{Introduction}

Minimal surfaces are mathematical abstractions of soap films. They are surfaces in $\R^3$ that locally minimize the area. In $1760$s Lagrange posed the problem of finding the surface of least area among all surfaces in $\R^3$ with a given closed boundary curve. Necessarily, a solution to this problem must be a minimal surface. It was shown by Monge in $1776$ that a surface is minimal if and only if its mean curvature is zero everywhere. It is also of importance to look at higher-dimensional analogs of minimal surfaces. These are submanifolds of an Euclidean space (or, more generally, of a Riemannian manifold) with zero mean curvature vector field, see Subsection \ref{sub_mean} for details. A nice account and survey of results about minimal surfaces and minimal submanifolds can be found in \cite{TF1987}.

Even though examples of minimal submanifolds are abound, there are not  so many of those that are defined by algebraic equations \cite{H1967}.

In \cite{Tka} Tkachev showed that the smooth locus $\M_{n,n,n-1}$ of the affine variety of singular real matrices of size $n\times n$ is a minimal hypersurface in the Euclidean space of all $n\times n$ matrices. To prove his result Tkachev shows that the determinant of a square matrix is an eigenfunction of the mean curvature operator, a condition known to be equivalent to the minimality of the associated hypersurface $\M_{n,n,n-1}\subset \{\det =0\}$.
In Theorem \ref{thm_general} we extend Tkachev's result to the semialgebraic submanifold $\M_{m,n,r}$ of $m\times n$ real matrices of rank $r$ and prove its minimality using a local parametrization of $\M_{m,n,r}$ coming from singular value decomposition \eqref{svd}.

In \cite{HLT} Hoppe, Linardopoulos and Turgut proved that the smooth locus $\Sk_{2n,2n-2}$ of the affine variety of singular real skew-symmetric matrices of size $2n\times 2n$ is a minimal hypersurface in the Euclidean space of all $2n\times 2n$ skew-symmetric matrices. The authors of \cite{HLT} show that the pfaffian of a skew-symmetric matrix of an even size is an eigenfunction of the mean curvature operator thus proving minimality of the hypersurface $\Sk_{2n,2n-2}\subset \{\textrm{pf} = 0\}$.
In Theorem \ref{thm_skew} we extend the result of Hoppe et al. to the semialgebraic submanifold $\Sk_{n,2r}$ of $n\times n$ real matrices of rank $2r$ and prove its minimality using a local parametrization of $\Sk_{n,2r}$ coming from the normal form \eqref{normalform} of a skew-symmetric matrix.

Finally, in Theorem \ref{thm_sym} we prove minimality of the set of real symmetric matrices whose eigenvalues have prescribed multiplicities.

\section{Main results}\label{sec:main}

For $m\leq n$ let $\M_{m,n}$ denote the space of $m\times n$ real matrices endowed with \emph{the Frobenius inner product} 
\begin{align}\label{Frob}
\langle A, B\rangle = \mathrm{Tr}(A^tB) = \sum_{i=1}^m\sum_{j=1}^n a_{ij}b_{ij},\quad A=(a_{ij}),\ B=(b_{ij})\in \M_{m,n}.
\end{align}

For a fixed $r\leq m\leq n$ denote by
\begin{align*}
\M_{m,n,r}  = \{A\in \M_{m,n}\, :\, \textrm{rank}(A)=r\} 
\end{align*}
the semialgebraic set of $m\times n$ matrices of rank $r$.
It is well-known that $\M_{m,n,r}$ is the smooth locus of the affine variety
\begin{align*}
  \bM_{m,n,r} = \{A\in \M_{m,n}\,:\,\textrm{rank}(A)\leq r\}
\end{align*}
of $m\times n$ matrices of rank at most $r$.

If $m=n$ and $r=n-1$ the variety
\begin{align*}
  \bM_{n,n,n-1} = \{A\in \M_{n,n}\, :\, \det(A) = 0\}
\end{align*}
of singular $n\times n$ matrices is a hypersurface in $\M_{n,n}$. In \cite{Tka} Tkachev proved that the smooth semialgebraic hypersurface $\M_{n,n,n-1}$ is minimal in $(\M_{n,n},\langle,\cdot,\cdot\rangle)$. A submanifold of a Riemannian manifold is said to be \emph{minimal} if its mean curvature vector field is identically zero, see Subsection \ref{sub_mean} for details. In our first main result we generalize this result to all \emph{determinantal submanifolds $\M_{m,n,r}$}.
\begin{theorem}\label{thm_general}
For $r\leq m\leq n$ the smooth semialgebraic set $\M_{m,n,r}$ is a minimal submanifold of $(\M_{m,n},\langle\cdot,\cdot\rangle)$. 
\end{theorem}
We give a proof of Theorem \ref{thm_general} in Subsection \ref{sub_general}. Our proof generalizes the proof of minimality of the $4$-dimensional submanifold $\M_{2,3,1}\subset \M_{2,3}$ of $2\times 3$ matrices of rank one given in \cite[p. $37$]{Hoppe}.

One can also consider the projective semialgebraic set $\mathbb{P}(\M_{m,n,r})\subset\mathbb{P}(\M_{m,n})$ of $m\times n$ real matrices of rank $r$. For this one endows the real projective space $\mathbb{P}(\M_{m,n})$ with the standard metric induced from \eqref{Frob}, see Subsection \ref{sub_minimality}. The following corollary is then implied by Theorem \ref{thm_general} and Proposition \ref{conical}.
\begin{corollary}\label{cor_general}
For $r\leq m\leq n$ the smooth projective semialgebraic set $\mathbb{P}(\M_{m,n,r})$ is a minimal submanifold of $\mathbb{P}(\M_{m,n})$.  
\end{corollary}

Let us denote by
\begin{align*}
  \Sk_n &= \{A=(a_{ij}) \in \M_{n,n}\, :\, a_{ij} = -a_{ji}\ \textrm{for any}\ i,j\} 
\end{align*}
the space of $n\times n$ real skew-symmetric matrices. The rank of a skew-symmetric matrix is even and, in particular, any skew-symmetric matrix of an odd size is singular. In the following, for $1\leq r\leq \floor{n/2}$ let
\begin{align*}
  \Sk_{n,2r} =\Sk_n\cap \M_{n,n,2r}=\{A\in \Sk_n\, :\, \textrm{rank}(A)=2r\}
\end{align*}
denote the semialgebraic set of $n\times n$ skew-symmetric matrices of rank $2r$. It is well-known that $\Sk_{n,2r}$ is the smooth locus of the \emph{skew-symmetric determinantal variety}
\begin{align*}
  \bSk_{n,2r} = \{A\in \Sk_n\,:\, \textrm{rank}(A)\leq 2r\}
\end{align*}
of $n\times n$ skew-symmetric matrices of rank at most $2r$.  If $r=n-1$ the variety
\begin{align*}
  \bSk_{2n,2n-2} = \{A\in \Sk_{2n}\,:\, \det(A) = 0\}
\end{align*}
of singular skew-symmetric matrices is a hypersurface in $\Sk_{2n}$ cut out by \emph{the pfaffian polynomial} that is defined via $\textrm{pf}(A) = \det(A)^2$, $A\in \Sk_{2n}$. 
In \cite{HLT} it was discovered that $\Sk_{2n,2n-2}$ is a minimal hypersurface in $\Sk_{2n}$. In our second main result we generalize this fact to all \emph{skew-symmetric determinantal submanifolds} $\Sk_{n,2r}$.
\begin{theorem}\label{thm_skew}
  For $1\leq r\leq \floor{n/2}$ the smooth semialgebraic set $\Sk_{n,2r}$ is a minimal submanifold of $(\Sk_n,\langle\cdot,\cdot\rangle)$.
\end{theorem}
We give a proof of Theorem \ref{thm_skew} in Subsection \ref{sub_skew}.

One can also consider the projective semialgebraic set $\mathbb{P}(\Sk_{n,2r})\subset\mathbb{P}(\Sk_{n})$ of $n\times n$ real skew-symmetric matrices of rank $r$. One again endows the real projective space $\mathbb{P}(\Sk_{n})$ with the standard metric induced from \eqref{Frob}, see Subsection \ref{sub_minimality}. The following corollary is then implied by Theorem \ref{thm_skew} and Proposition \ref{conical}.
\begin{corollary}\label{cor_skew}
  For $1\leq r\leq \floor{n/2}$ the smooth projective semialgebraic set $\mathbb{P}(\Sk_{n,2r})$ is a minimal submanifold of $\mathbb{P}(\Sk_n)$.
\end{corollary}

Let us denote by
\begin{align*}
  \Sym_n = \{A=(a_{ij}) \in \M_{n,n}\, :\, a_{ij} = a_{ji}\ \textrm{for any}\ i,j\}
\end{align*}
the space of $n\times n$ real symmetric matrices.
\begin{remark}
  Theorems \ref{thm_general} and \ref{thm_skew} may suggest that \emph{the symmetric determinantal submanifold} $\M_{n,n,r}\cap \Sym_n$ of symmetric matrices of rank $r$ is minimal in $(\Sym_n,\langle\cdot,\cdot\rangle)$. However, it is not in general the case. For example, it is easy to see that the surface
  \begin{align*}
   \left\{a=
    \begin{pmatrix}
      a_{11} & a_{12} \\
      a_{12} & a_{22}
    \end{pmatrix}\,:\, \det(a)=a_{11}a_{22}-a_{12}^2=0\right\}\subset \Sym_2=\R^3
  \end{align*}
  of singular $2\times 2$ real symmetric matrices has non-zero mean curvature.   
\end{remark}

Given a real symmetric matrix $A\in \Sym_n$ let us denote by $\chi_A(t)=\det(t\,\textrm{Id}-A)$ its characteristic polynomial. An eigenvalue $\lambda$ of $A\in \Sym_n$ has \emph{multiplicity} $m$, where $1\leq m\leq n$, if $\chi^{(i)}_A(\lambda)=0$ for $i=0,\dots,m-1$ and $\chi_A^{(m)}(\lambda)\neq 0$.

For a vector $\vec{\kappa}=(\kappa_1,\kappa_2,\dots)$ of non-negative integers such that $1\kappa_1+2\kappa_2+\dots = n$ let us denote by
\begin{align*}
\Sym_{n,\vec{\kappa}} = \{A\in \Sym_n\, :\, \textrm{$A$ has $\kappa_i$ eigenvalues of multiplicity $i$}\}
\end{align*}
the semialgebraic set of $n\times n$ real symmetric matrices that have $\kappa_i$ eigenvalues of multiplicity $i$, $i\geq 1$. Sets $\Sym_{n,\vec\kappa}$ are smooth submanifolds of $\Sym_{n}$ and they form a stratification of $\Sym_n$ with $\Sym_{n,(n)}$ being the unique open stratum \cite{Arnold1972}. They were studied in \cite{Arnold1972, Agrachev, BKL}. We discover a new fact about $\Sym_{n,\vec\kappa}$, namely its minimality.
\begin{theorem}\label{thm_sym}
 For any vector $\vec\kappa=(\kappa_1,\kappa_2,\dots)$ the smooth semialgebraic set $\Sym_{n,\vec\kappa}$ is a minimal submanifold of $(\Sym_n,\langle\cdot,\cdot\rangle)$. 
\end{theorem}
We give a proof of Theorem \ref{thm_sym} in Subsection \ref{sub_sym}.
One can again consider the projective version $\mathbb{P}(\Sym_{n,\vec\kappa})\subset\mathbb{P}(\Sym_{n})$ consisting of $n\times n$ real symmetric matrices with $\kappa_i$ eigenvalues of multiplicity $i$, $i\geq 1$. The real projective space $\mathbb{P}(\Sym_{n})$ is endowed with the standard metric induced from \eqref{Frob}, see Subsection \ref{sub_minimality}. Theorem \ref{thm_sym} and Proposition \ref{conical} imply the following corollary.
\begin{corollary}\label{cor_sym}
  For any vector $\vec\kappa=(\kappa_1,\kappa_2,\dots)$ the smooth projective semialgebraic set $\mathbb{P}(\Sym_{n,\vec\kappa})$ is a minimal submanifold of $\mathbb{P}(\Sym_n)$.
\end{corollary}

\section{Preliminaries and auxiliary results}\label{sec:prel}

In this section we state some facts and results that we then use in Section \ref{sec:proofs} to prove our main results.
\subsection{Singular value decomposition}
Let $O(n) = \{V\in \M_{n,n}: V^\mathsf{T}V=\mathrm{id}\}$ denote the group of orthogonal $n\times n$ matrices. The standard action of the product $O(m)\times O(n)$ of orthogonal groups on $\M_{m,n}$,
\begin{align}\label{action}
  (U,V)\in O(m)\times O(n),\  A\in \M_{m,n}\ \mapsto\ UAV^\mathsf{T},
\end{align}
preserves the Frobenius inner product \eqref{Frob}, that is, for $(U,V)\in O(m)\times O(n)$
\begin{align*}
\langle UAV^\mathsf{T}, UBV^\mathsf{T}\rangle =\mathrm{Tr}(VA^\mathsf{T}U^\mathsf{T}UBV^\mathsf{T})= \mathrm{Tr}(A^\mathsf{T}B) =  \langle A,B\rangle,\quad A, B\in \M_{m,n}.
\end{align*}
Moreover, the action \eqref{action} obviously preserves the rank of a matrix and hence the manifold $\M_{m,n,r}$ is invariant under $O(m)\times O(n)$ for any $r\leq m\leq n$. \emph{The singular value decomposition} (in the following SVD) of a matrix $A\in \M_{m,n}$ is a factorization
\begin{align}\label{svd}
  A=U^\mathsf{T}\Sigma V,
\end{align}
where $U\in O(m)$, $V\in O(n)$ and
\begin{align*}
  \Sigma =
  \begin{pmatrix}[cccccc]
    \sigma_1 &          & \text{\large 0} & &               &   \\
             &  \ddots  &                 & &\text{\Huge 0} &   \\
             \text{\large 0}   &    & \sigma_m & & & 
  \end{pmatrix}
\end{align*}
is the ``diagonal'' matrix of \emph{singular values} $\sigma_1,\dots,\sigma_m\geq 0$ of $A$. Note that the number of non-zero singular values equals the rank of $A$ and we can, without loss of generality, assume that they are ordered, $\sigma_1\geq \dots\geq \sigma_m\geq 0$.

In Subsection \ref{sub_general} we use the singular value decomposition \eqref{svd} of a matrix in order to design a local parametrization of $\M_{m,n,r}$ suitable for computing its mean curvature.

\subsection{Normal form of a skew-symmetric matrix}

Consider the following diagonal subaction of the action \eqref{action} on the space $\Sk_n$ of $n\times n$ real skew-symmetric matrices
\begin{align}\label{skaction}
  V\in O(n),\ A\in \Sk_n\ \mapsto\ VAV^\mathsf{T}. 
\end{align}
Any $A\in \Sk_n$ can be written in \emph{the normal form} \cite[Thm. $2.5$]{Thompson}
\begin{align}\label{normalform}
  A=V^\mathsf{T}\Omega V,
\end{align}
where $V\in O(n)$ and 
\begin{align}\label{omega}
 \Omega =  \begin{pmatrix}
    0         & \omega_1 &         &        &          &        &        & \\
    -\omega_1 & 0        &         &        &          &         \text{\Huge 0} &         &  \\
             &           & \ddots &          &           &         &      & \\
              &         &         & 0        &  \omega_r &         &        & \\
              &         &        & -\omega_r & 0     &         &        & \\
              &         \text{\Huge 0}&         &           &       &  0      &        & \\
              &         &         &           &       &         &\ddots  & \\
              &         &        &           &       &         &        & 0
  \end{pmatrix},
\end{align}
where $\pm i\, \omega_1,\dots, \pm i\,\omega_r\in i\cdot \R$ are non-zero eigenvalues of $A$ and, in particular, $2r$ is the rank of $A$. We can, moreover, assume that $\omega_1\geq \dots \geq \omega_r>0$ are positive and ordered.

In Subsection \ref{sub_skew} we use the normal form \eqref{normalform} of a skew-symmetric matrix in order to design a local parametrization of $\Sk_{n,2r}$ suitable for computing its mean curvature.

\subsection{Spectral decomposition of a symmetric matrix}

Consider the following diagonal subaction of the action \eqref{action} on the space $\Sym_n$ of $n\times n$ real symmetric matrices
\begin{align}\label{symaction}
  V\in O(n),\ A\in \Sym_n\ \mapsto\ VAV^\mathsf{T}. 
\end{align}
\emph{The spectral decomposition of a matrix $A\in \Sym_n$} is a factorization
\begin{align}\label{specdec}
  A=V^\mathsf{T}\Lambda V,
\end{align}
where $V\in O(n)$ and 
  \begin{equation*}
\Lambda =   \begin{pmatrix}
  \lambda_1 & &\textrm{\Large 0} \\
  & \ddots &  \\
 \textrm{\Large 0} & & \lambda_n
\end{pmatrix}
\end{equation*}
is the diagonal matrix of \emph{eigenvalues} $\lambda_1,\dots\lambda_n\in \R$ of $A$. Note that singular values of $A$ are related to its eigenvalues via $\sigma_i=\vert\lambda_i\vert$, $i=1,\dots,n$.

Semialgebraic manifolds $\Sym_{n,\vec\kappa}$ are invariant under the action \eqref{symaction}. It is not difficult to show that the membership of $A\in \Sym_n$ in a certain $\Sym_{n,\kappa}$ is determined by the conjugacy class of its stabilizer under \eqref{symaction}. To prove this, we need the following lemma.
\begin{lemma}\label{stab}
  Let $A\in \Sym_{n,\vec\kappa}$. The stabilizer $O(n)_A$ of $A$ under the action \eqref{symaction} equals
  \begin{align*}
    O(n)_A = V^\mathsf{T} O_{\vec\kappa} V,
  \end{align*}
  where $V\in O(n)$ and
  \begin{align*}
    O_{\vec\kappa} = \bigtimes\limits_{i\geq 1} O(i)^{\,\kappa_i} = \bigtimes_{i\geq 1} \underbrace{O(i)\times\dots\times O(i)}_{\kappa_i\ \textrm{times}}
  \end{align*}
  is the direct product of orthogonal groups of sizes encoded by the partition $\vec\kappa$.
\end{lemma}
\begin{proof}
  By \eqref{specdec} we can write $A=V^\mathsf{T}\Lambda V$ for some $V\in O(n)$ and diagonal matrix $\Lambda$ of eigenvalues. Since $A\in \Sym_{n,\vec\kappa}$ there are $\vert\vec\kappa\vert = \kappa_1+\kappa_2+\dots$ pairwise distinct eigenvalues that we denote by $\lambda^{(i)}_j$, $i\geq 1$, $j=1,\dots, \kappa_i$. After a possible permutation of the entries of $\Lambda$, we can assume that it is of the form
  
  \begin{equation*}
    \Lambda =
    \begin{pmatrix}
      \ddots & &\coolover{i}{\vspace{0.2cm}\phantom{\lambda_j} & \phantom{\ddots} & \phantom{\lambda_j}} &  \\
      &      & \lambda_j^{(i)} & &\textrm{\huge 0} & \\
      &      &                &\ddots & & \\
      &      &  \textrm{\huge 0}              &       & \lambda^{(i)}_j &\\
      &      &                &       &        & \ddots
    \end{pmatrix}
    \coolrightbrace{i}{\\ \\ \\}\, ,
  \end{equation*}
  namely the diagonal entries of $\Lambda$ are $\lambda^{(1)}_1,\dots,\lambda^{(1)}_{\kappa_1},\lambda^{(2)}_{1},\lambda^{(2)}_1, \dots, \lambda^{(2)}_{\kappa_2},\lambda^{(2)}_{\kappa_2},\lambda^{(3)}_1,\lambda^{(3)}_1,\lambda^{(3)}_1,\dots$.

Let $U\in O(n)_A$ be an orthogonal matrix that fixes $A$, that is, $UV^\mathsf{T}\Lambda VU^\mathsf{T} = UAU^\mathsf{T} = A = V^\mathsf{T}\Lambda V$ or, equivalently, $VUV^\mathsf{T} \Lambda = \Lambda VUV^\mathsf{T}$. Denoting $W=VUV^\mathsf{T}$ and taking the $(i,j)$th entry of $W\Lambda=\Lambda W$ we obtain
\begin{align*}
  W_{ij}\Lambda_{jj} = \Lambda_{ii}W_{ij}.
\end{align*}
Thus, if $\Lambda_{ii}\neq \Lambda_{jj}$, then $W_{ij}=0$. Due to the block-diagonal structure of $\Lambda$ with blocks being scalar matrices $\lambda_j^{(i)} \mathrm{Id}_{\,i\times i}$ with pairwise distinct $\lambda_j^{(i)}$, the orthogonal matrix $W\in O(n)$ is block-diagonal with the same block structure as in $\Lambda$. The condition $W^\mathsf{T}W=\mathrm{Id}$ implies that each $i\times i$ block is a (small) orthogonal matrix. Therefore, $W\in O_{\vec\kappa}= \bigtimes_{\,i\geq 1} O(i)^{\,\kappa_i}$ and $U=V^\mathsf{T}WV \in V^\mathsf{T} O_{\vec\kappa}V$. Conversely, any $U\in V^\mathsf{T} O_{\vec\kappa}V$ obviously fixes $A\in \Sym_{n,\vec\kappa}$ written in the above form.
\end{proof}
\begin{corollary}\label{conj}
A real symmetric matrix $A\in \Sym_n$ is in $\Sym_{n,\vec\kappa}$ if and only if its stablizer $O(n)_A$ under the action \eqref{symaction} is conjugate to $O_{\vec\kappa}=\bigtimes_{i\geq 1} O(i)^{\,\kappa_i}$.  
\end{corollary}
This characterization of $\Sym_{n,\vec\kappa}$ is used in the proof of Theorem \ref{thm_sym} in Subsection \ref{sub_sym}.

\subsection{Mean curvature of a submanifold of a Riemannian manifold}\label{sub_mean}
In this subsection we very briefly recall a definition of the mean curvature vector field of a submanifold of a Riemannian manifold, see \cite[Ch. $\text{\MakeUppercase{\romannumeral 7}}$]{KN1969} for more details.

Let $(M,\bm g)$ be a Riemannian manifold and let $X\subset M$ be its smooth $n$-dimensional submanifold. \emph{The second fundamental form of $X\subset M$} is a symmetric bilinear form $\bm b$ on the tangent bundle $TX$ to $X$ with values in the normal bundle $(TX)^\perp$ to $X$ defined at each point $p\in X$ by
\begin{align*}
  \bm b: T_{p\,}X\times T_{p\,}X &\ \rightarrow\ (T_{p\,}X)^\perp,\\
       \xi, \eta\hspace{0.7cm} &\ \mapsto\ (\nabla_\xi \,\vec{\eta}\,)^{\perp},
\end{align*}
where $\vec{\eta}$ is a smooth vector field defined on some neighborhood $U\subset M$ of $p$ and such that $\vec{\eta}_{\,p}=\eta$, $\nabla_\xi\,\vec{\eta}\in~T_{p\,}M$ is the Levi-Civita covariant derivative of $\vec{\eta}$ along the vector $\xi\in T_{p\,}X$ and $(\nabla_\xi\,\vec{\eta}\,)^\perp\in~(T_{p\,}X)^\perp$ is \emph{the normal component of $\nabla_\xi\,\vec{\eta} \in T_{p\,}M$}. The result $(\nabla_\xi\, \vec{\eta}\,)^\perp$ is independent of the choice of a vector field $\vec{\eta}$ that extends $\eta\in T_{p\,}X$. Moreover, if $\vec{\xi}$ and $\vec{\eta}$ are smooth vector fields tangent to $X$ along some open set $U\subset X$, the dependence
\begin{align*}
  p\in U\ \mapsto\  \bm b(\vec{\xi}_{\,p},\vec{\eta}_{\,p})\in (T_{p\,}X)^\perp
\end{align*}
is a smooth field of normal vectors to $X$.

Let $\vec{e}=\{\vec{e}^{\;1},\dots, \vec{e}^{\;n}\}$ be a \emph{local frame on $X$}, that is, $\vec{e}^{\;1},\dots,\vec{e}^{\;n}$ are smooth vector fields tangent to $X$ along some open set $U\subset X$ and such that for each $p\in U$ vectors $\vec{e}^{\;1}_p,\dots, \vec{e}^{\;n}_p\in T_{p\,}X$ form a basis of the tangent space $T_{p\,}X$. Let $\vec{G}=(\bm g(\vec{e}^{\;i},\vec{e}^{\;j}))$ be the smooth field of matrices of the metric $\bm g$ written in the local frame $\vec{e}$ and let $\vec{G}^{-1}$ be the smooth field of inverses of $\vec{G}$, that is, $\vec{G}^{\,-1}_p = (\bm g(\vec{e}_p^{\; i},\vec{e}_p^{\;j}))^{-1}$ for $p\in U$. Then \emph{the mean curvature vector field of $X$ along $U$} is defined by
\begin{align}\label{mean_loc}
  H\vert_U = \sum_{i,j=1}^n(\vec{G}^{\,-1})_{ij}\, \bm b( \vec{e}^{\; i},\vec{e}^{\;j}),
\end{align}
where $\bm b(\vec{e}^{\; i},\vec{e}^{\;j})$ is the field of normal vectors $\bm b(\vec{e}_p^{\; i},\vec{e}_p^{\;j})\in (T_{p\,}X)^\perp$, $p\in U$, and \emph{the mean curvature vector of $X$ at a point $p\in U$} is given by
\begin{align*}
  H_{p\,} = \sum_{i, j=1}^n (\vec{G}^{\,-1}_p)_{ij} \,\bm b( \vec{e}_{p}^{\; i},\vec{e}_p^{\;j}) \in (T_{p\,}X)^\perp.
\end{align*}
The definition \eqref{mean_loc} of $H\vert_U$ is independent of the choice of a local frame on $U$. By gluing mean curvature vector fields \eqref{mean_loc} along open sets from an open cover of $X$ we obtain the smooth field $H$ of normal vectors to $X$, called \emph{the mean curvature vector field of $X$}. A submanifold $X\subset M$ is called \emph{minimal} if its mean curvature vector field is zero.
\begin{remark}
  Usually one defines the mean curvature vector field of an $n$-dimensional submanifold $X\subset M$ as $\frac{1}{n} H$, where $H$ is defined above. However, since we are interested in minimal submanifolds, that is, when $H=0$, the factor of $\frac{1}{n}$ is unessential for us.  
\end{remark}
\begin{example}
  If $X\subset (\R^N,\langle \cdot,\cdot\rangle)$ is a submanifold of the Euclidean space, the covariant derivative $\nabla$ coincides with the directional derivative of components of a vector field along a vector and, in particular, the second fundamental form of $X$ computed on two vectors $\xi=(\xi^{\,1},\dots,\xi^{\,N})$, $\eta=(\eta^{\,1},\dots,\eta^{\,N})\in T_{p\,}X\subset \R^N$ equals
\begin{align}\label{b}
  \bm b(\xi,\eta)  = \left(\sum_{i=1}^N\xi^{\,i} \partial_{x_i}\vec{\eta}^{\,1},\dots,\sum_{i=1}^N\xi^{\,i}\partial_{x_i}\vec{\eta}^{\,N}\right)^\perp \in (T_{p\,}X)^\perp \subset \R^N,
\end{align}
where $\vec{\eta} = (\vec{\eta}^{\,1},\dots,\vec{\eta}^{\,N})$ is a smooth vector field defined in a neighborhood of $p$ and such that $\vec\eta_{\,p}=\eta\in T_{p\,}X$.
Consider \emph{a local parametrization of $X$}, that is, a smooth map
\begin{align}\label{r}
  r: U\subset \R^n\ \rightarrow\ X
\end{align}
from some open set $U\subset \R^n$ to $X$ such that the vector fields $\partial_{u_1}r,\dots, \partial_{u_n}r$ form a local frame on $X$ along $U$. It follows from \eqref{b} that
\begin{align}\label{b_eucl}
\bm b(\partial_{u_i}r,\partial_{u_j}r) = \left(\partial_{u_i}\partial_{u_j} r\right)^\perp,\quad i, j=1,\dots, n. 
\end{align}
In particular, the mean curvature vector of $X$ at $r(u)\in X$ is computed as
\begin{align}\label{mean}
H_{r(u)} = \sum_{i,j=1}^n (\vec{G}_{r(u)}^{\,-1})_{ij} (\partial_{u_i}\partial_{u_j} r(u))^\perp \in (T_{r(u)}X)^\perp, 
\end{align}
where $\vec{G}^{\,-1}_{r(u)}$ is the inverse of the matrix $\vec{G}_{r(u)} = (\langle \partial_{u_i}r(u),\partial_{u_j}r(u)\rangle)$ of the metric written in the basis $\partial_{u_1}r(u),\dots, \partial_{u_n}r(u)\in T_{r(u)}X$. Denoting by $\textrm{d}^2r(u)=(\partial_{u_i}\partial_{u_j} r(u))$ the $n\times n$ matrix of second partial derivatives of the local parametrization \eqref{r}, we formally write 
\begin{align}\label{mean2}
  H_{r(u)} = \textrm{Tr}\left[\vec{G}_{r(u)}^{\,-1} \,(\textrm{d}^2 r(u))^\perp\right],
\end{align}
where $(\textrm{d}^2r(u))^\perp=((\partial_{u_i}\partial_{u_j}r(u))^\perp)$.

If $X\subset M\subset \R^N$, expressions \eqref{b_eucl}, \eqref{mean} and \eqref{mean2} are valid with a slight change: additionally, one needs to project vectors $\left(\partial_{u_i}\partial_{u_j}r(u)\right)^\perp$ to the tangent space $T_{r(u)}M$.
\end{example}
In our computation of mean curvature vectors of $\M_{m,n,r}$ and $\Sk_{n,2r}$ in Subsections \ref{sub_general} and \ref{sub_skew} it is more convenient to use the formal form \eqref{mean2} of the expression \eqref{mean}. 

\subsection{Minimality in Euclidean and projective spaces}\label{sub_minimality}

In this subsection we give a proof of the folklore fact that a conic submanifold of an Euclidean space is minimal if and only if its intersection with a sphere is a minimal submanifold of the sphere if and only if its projectivization is a minimal submanifold of the real projective space.

Let $(\R^N,\langle\cdot,\cdot\rangle)$ be an Euclidean space and denote by $S^{N-1}=\{p\in \R^N:  \langle p,p\rangle=1\}$ \emph{the unit sphere} in $\R^N$ endowed with the induced metric. Let $\mathbb{P}(\R^N)$ be \emph{the projective $(N-1)$-space}. \emph{The standard metric on $\mathbb{P}(\R^N)$} is the push-forward metric under the double covering map $S^{N-1}\rightarrow \mathbb{P}(\R^N)$ that sends $p\in S^{N-1}$ to the line through $p$ and $-p$.

A smooth submanifold $X\subset \R^N$ is said to be \emph{conic} if $tp\in P$ for any $p\in X$ and any $t\in \R\setminus\{0\}$. Given a conic submanifold $X\subset \R^N$, its projectivization $\mathbb{P}(X)\subset \mathbb{P}(\R^N)$ is a smooth submanifold of the projective space $\mathbb{P}(\R^N)$.

\begin{example}
The manifold $\M_{m,n,r}$ of $m\times n$ real matrices of rank $r$, the manifold $\Sk_{n,2r}$ of $n\times n$ real skew-symmetric matrices of rank $2r$ and the manifold $\Sym_{n,\vec{\kappa}}$ of $n\times n$ real symmetric matrices with exactly $\kappa_i$ eigenvalues of multiplicity $i$, where $\vec{\kappa}=(\kappa_1,\kappa_2,\dots)$, are conic submanifolds of $\M_{m,n}$, $\Sk_n$ and $\Sym_{n}$ respectively.
\end{example}

The following fact is well-known, but we anyway include a proof of it.
\begin{proposition}\label{conical}
    Let $X\subset \R^N$ be a conic submanifold. Then $X$ is minimal in $\R^N$ if and only if $X\cap S^{N-1}$ is minimal in $S^{N-1}$ if and only if $\mathbb{P}(X)$ is minimal in $\mathbb{P}(\R^N)$.
\end{proposition}
\begin{proof}
  Observe first that for any $t\in \R\setminus \{0\}$ the dilation
  \begin{align*}
    \delta_{\,t} :\R^N&\ \rightarrow\ \R^N,\\
p&\ \mapsto\ t\hspace{0.01cm}p
  \end{align*}
  is a homothety of $(\R^N,\langle \cdot,\cdot\rangle)$. Thus to prove that a conic submanifold $X\subset \R^N$ is minimal it is enough to show that the mean curvature vector $H_p$ of $X$ is zero at any $p\in X\cap S^{N-1}$. Now, the tangent space to the sphere at $p\in S^{N-1}$ is identified with the space of vectors orthogonal to $p$, that is, $T_{p\,}S^{N-1} = p^{\,\perp} = \{\xi\in \R^N: \langle \xi,p\rangle = 0\}$. Under this identification and in view of the fact that $p\in T_{p\,}X$, at any $p\in X\cap S^{N-1}$ the normal spaces to $X\subset \R^N$ and to $X\cap S^{N-1}\subset S^{N-1}$ coincide, $(T_{p\,}X)^\perp = (T_{p\,} (X\cap S^{N-1}))^\perp\subset p^{\,\perp}\subset \R^N$.

  Assume that $X$ has dimension $n$ and consider a local parametrization $r: U\rightarrow X\cap S^{N-1}$ of $X\cap S^{N-1}$ near $p$ such that $0\in U\subset \R^{n-1}$ and $r(0)=p$. Then the map
  \begin{align*}
  R: U\times (-\varepsilon,\varepsilon)\ &\rightarrow\ X,\\
        u=(u_1,\dots,u_{n-1}), u_n\ &\mapsto\ (1+u_n) r(u),
\end{align*}
is a local parametrization of $X$ near $p$ such that $R(0)=p$.
The metric $G=\left(\langle \partial_{u_i}R(0),\partial_{u_j}R(0)\rangle\right)$ written in the basis
\begin{align}\label{basis_X}
\partial_{u_1}R(0)=\partial_{u_1}r(0),\dots, \partial_{u_{n-1}}R(0)=\partial_{u_{n-1}}r(0), \partial_{u_n}R(0)=p
\end{align}
has a block-diagonal form
\begin{equation*}
G=
\begin{pmatrix}[ccc|c]
    & & & 0 \\
    & \textrm{\large $\left(\langle \partial_{u_i}r(0), \partial_{u_j}r(0)\rangle\right)$}& & \vdots  \\
    & & & 0\\ \hline
   0 & \dots & 0 & 1 \\
  \end{pmatrix},
\end{equation*}
where the upper-left block $g=\left(\langle \partial_{u_i}r(0),\partial_{u_j} r(0)\rangle\right)$ is the metric on $T_{p\,}(X\cap S^{N-1})$.

Since $\partial_{u_n}\partial_{u_n}R(0) = 0$ and $\partial_{u_i}\partial_{u_n}R(0) = \partial_{u_i}r(0)\in T_{p\,}(X\cap S^{N-1})$, $i=1,\dots, n-1$, formula \eqref{b_eucl} implies that the matrix of the second fundamental form of $X\subset \R^N$ at $p$ computed in the basis \eqref{basis_X} has the form
\begin{equation}\label{b(X)}
\left(\textrm{d}^2R(0)\right)^\perp =  \begin{pmatrix}[ccc|c]
     & & & 0 \\
    & \textrm{\large $\left(\left(\partial_{u_i}\partial_{u_j}r(0)\right)^\perp\right)$}& & \vdots  \\
    & & & 0\\ \hline
   0 & \dots & 0 & 0 \\
  \end{pmatrix}.
\end{equation}
Since the normal spaces to $X$ and $X\cap S^{N-1}$ at $p$ coincide,  $\left(\partial_{u_i}\partial_{u_j}r(0)\right)^\perp\in (T_{p\,}X)^\perp=(T_{p\,}(X\cap S^{N-1}))^\perp$ is the value of the second fundamental form of both $X\subset \R^N$ and $X\cap S^{N-1}\subset S^{N-1}$.
Thus the upper-left block in \eqref{b(X)} is the matrix $\left(\textrm{d}^2 r(0)\right)^\perp=\left(\left(\partial_{u_i}\partial_{u_j}r(0)\right)^\perp\right)$ of the second fundamental form of $X\cap S^{N-1}\subset S^{N-1}$ at $p$ written in the basis $\partial_{u_1}r(0), \dots, \partial_{u_{n-1}}r(0)$ of $T_{p\,}(X\cap S^{N-1})$.
As a consequence, the mean curvature vectors of $X\subset \R^N$ and $X\cap S^{N-1}\subset S^{N-1}$ at $p$ are equal,
\begin{align*}
  H(X)_p=\textrm{Tr}\left( G^{-1}\left(\textrm{d}^2R(0)\right)^\perp\right) = \textrm{Tr}\left(g^{-1}\left(\textrm{d}^2r(0)\right)^\perp\right) = H(X\cap S^{N-1})_p.
\end{align*}
In particular, a conic submanifold $X\subset \R^N$ is minimal if and only if its intersection with the sphere $X\cap S^{N-1}$ is minimal in $S^{N-1}$. Finally, since the double covering $S^{N-1}\rightarrow \mathbb{P}(\R^N)$ is, by construction, a local isometry, and since the definition of the mean curvature is also local, the second equivalence in the statement of the proposition follows. 
\end{proof}

In Section \ref{sec:main} we use Proposition \ref{conical} to derive Corollaries \ref{cor_general}, \ref{cor_skew} and \ref{cor_sym} from Theorems \ref{thm_general}, \ref{thm_skew} and \ref{thm_sym} respectively.

\section{Proof of main results}\label{sec:proofs}

In this section we prove Theorems \ref{thm_general}, \ref{thm_skew} and \ref{thm_sym}. Proofs of Theorems \ref{thm_general} and \ref{thm_skew} are performed in local coordinates given by normal forms \eqref{svd} and \eqref{normalform}, while the proof of Theorem \ref{thm_sym} is derived from a general result of Hsiang and Lawson from \cite{HL1971}.
\subsection{Proof of Theorem \ref{thm_general}}\label{sub_general}
We write a matrix $A\in \M_{m,n,r}$ in the SVD form \eqref{svd} $A=U\Sigma V^\mathsf{T}$, where $U\in O(m)$, $V\in O(n)$ and
% \begin{align}
%   \Sigma =
%   \begin{pmatrix}
%     \sigma_1 & \dots & . & \dots & .  \\
%     \vdots & \ddots & \vdots & \vdots & \vdots\\
%     . & \dots & \sigma_r & \vdots & \vdots\\
%     . & \dots & . & \dots& . 
%   \end{pmatrix}
% \end{align}
\begin{align*}
  \Sigma =
  \begin{pmatrix}
    \sigma_1 &        &         &      &          &     & &  &\\
             & \ddots &         &      &    \text{\Large 0}    & &  &       &\\
             &        &\sigma_r &      &          &       & & &\\
             &        &         & 0    &         &       & &\text{{\Huge 0}} &\\
             &      \text{\Large 0}  &         &      & \ddots  & & &       &\\
             &        &         &      &         &  0    &  & & 
  \end{pmatrix}
\end{align*}
is an $m\times n$ diagonal matrix of singular values of $A$. We first assume that the nonzero singular values are ordered and distinct, that is,  $\sigma_1> \dots > \sigma_r>0$. Recall from Section~\ref{sec:prel} that the inner product \eqref{Frob} is invariant under the action \eqref{action} of $O(m)\times O(n)$. Therefore, to prove that $\M_{m,n,r}$ is a minimal submanifold of $(\M_{m,n},\langle\cdot,\cdot\rangle)$ it is enough to show that for any diagonal matrix $A=\Sigma\in \M_{m,n,r}$ the mean curvature vector \eqref{mean2} at $A=\Sigma$ is zero. For this we consider the following parametrization of a neighborhood of $A=\Sigma$:
\begin{align}\label{par}
  A(\bm{\mu},\bm{s},\bm\nu) = \left(\prod_{\substack{1\leq i<j\leq m,\\ i\leq r}} e^{\mu_{ij} L_{ij}}\right)\left(\Sigma+\sum_{h=1}^r s_hE_{hh}\right)\left(\prod_{\substack{1\leq k<\ell\leq n,\\ k\leq r}}e^{-\nu_{k\ell} L_{k\ell}}\right),
\end{align}
where $E_{ij}$ denotes the $(i,j)$th matrix unit, $L_{ij}=E_{ji}-E_{ij}$ and matrices $e^{\mu_{ij}L_{ij}}$ and $e^{-\nu_{k\ell}L_{k\ell}}$ in the two products are ordered according to the lexicographic order on sets of indices $(i,j)$ and $(k,\ell)$. Note that $\mu\mapsto e^{\mu L_{ij}}$ is a smooth one-parameter subgroup of orthogonal matrices such that $e^{0 L_{ij}} = \mathrm{id}$ and $\frac{d}{d\mu} e^{\mu L_{ij}} = L_{ij}e^{\mu L_{ij}} = e^{\mu L_{ij}} L_{ij}$. In particular, $A(\bm 0,\bm 0,\bm 0) = A=\Sigma$ and $\frac{d}{d\mu}\big|_{\mu=0} e^{\mu L_{ij}} = L_{ij}$.
Using this we now compute first order derivatives of the parametrization \eqref{par}. We have
\begin{equation}\label{der_mu}
\footnotesize\begin{aligned}
  \partial_{\mu_{ij}} A(\bm\mu,\bm s,\bm\nu) = \left( e^{\mu_{12}L_{12}}\cdots e^{\mu_{ij}L_{ij}} L_{ij} \cdots e^{\mu_{rm}L_{rm}}\right)\left(\Sigma+\sum_{h=1}^r s_hE_{hh}\right)\left(\prod_{\substack{1\leq k<\ell\leq n,\\ k\leq r}}e^{-\nu_{k\ell} L_{k\ell}}\right)  
\end{aligned}
\end{equation}
for $1\leq i<j\leq m$ with $i\leq r$,
\begin{equation}\label{der_s}
\begin{aligned}
\partial_{s_h}A(\bm\mu,\bm s,\bm \nu) = \left(\prod_{\substack{1\leq i<j\leq m,\\ i\leq r}} e^{\mu_{ij} L_{ij}}\right)E_{hh}\left(\prod_{\substack{1\leq k<\ell\leq n,\\ k\leq r}}e^{-\nu_{k\ell} L_{k\ell}}\right)
\end{aligned}
\end{equation}
for $1\leq h\leq r$ and
\begin{equation}\label{der_nu}
  \footnotesize\begin{aligned}
    \partial_{\nu_{k\ell}} A(\bm\mu,\bm s,\bm\nu) = \left(\prod_{\substack{1\leq i<j\leq m,\\ i\leq r}} e^{\mu_{ij} L_{ij}}\right)\left(\Sigma+\sum_{h=1}^r s_hE_{hh}\right)\left( e^{-\nu_{12}L_{12}}\cdots (-L_{k\ell})e^{-\mu_{k\ell}L_{k\ell}} \cdots e^{-\mu_{rn}L_{rn}}\right)  
   \end{aligned}
\end{equation}
for $1\leq k<\ell\leq n$ with $k\leq r$. Note that matrices \eqref{der_mu}, \eqref{der_s} and \eqref{der_nu} belong to the tangent space to $\M_{m,n,r}$ at the point $A(\bm\mu,\bm s,\bm\nu)$. At $A(\bm 0):=A(\bm 0,\bm 0,\bm 0)=A=\Sigma$ these are equal
\begin{equation}\label{basis}
\begin{aligned}
  \partial_{\mu_{ij}}A(\bm 0) &= L_{ij}\Sigma = \sigma_i E_{ji}-\sigma_j E_{ij},\ &1\leq i<j\leq m,\ i\leq r,\\
  \partial_{s_h}A(\bm 0) &=E_{hh},\ &1\leq h\leq r,\\
  \partial_{\nu_{k\ell}}A(\bm 0) &= \Sigma(- L_{k\ell}) = \sigma_k E_{k\ell}-\sigma_\ell E_{\ell k},\ &1\leq k<\ell\leq n,\ k\leq r,                                     
\end{aligned}
\end{equation}
where we set $\sigma_j=\sigma_\ell=0$ for $j, \ell >r$. There are
\begin{align}\label{count}
  {m \choose 2} - {m-r \choose 2} + r + {n \choose 2} - {n-r \choose 2} = (m+n)r-r^2
\end{align}
matrices in \eqref{basis} and it is easy to see that they are linearly independent. The count \eqref{count} and the formula $\dim(\M_{m,n,r}) = (m+n)r-r^2$ \cite{Harris} imply that \eqref{par} is indeed a parametrization of $\M_{m,n,r}$ around $A(\bm 0)=A$ and, in particular, matrices \eqref{basis} form a basis of the tangent space to $\M_{m,n,r}$ at $A=\Sigma$. We now compute the metric tensor of $\M_{m,n,r}$ at $A$ in this basis. Let us observe that matrices $\partial_{\mu_{ij}} A(\bm 0)$ are orthogonal among themselves. The same holds for $\partial_{s_h}A(\bm 0)$ and for $\partial_{\nu_{k\ell}}A(\bm 0)$. Furthermore, matrices $\partial_{\mu_{ij}} A(\bm 0)$ any $\partial_{\nu_{k\ell}}A(\bm 0)$ are orthogonal to $\partial_{s_h} A(\bm 0)$. Finally, $\partial_{\mu_{ij}} A(\bm 0)$ is orthogonal to $\partial_{\nu_{k\ell}} A(\bm 0)$ unless $i=k$ and $j=\ell\leq r$ in which case their inner product equals $\langle \partial_{\mu_{ij}}A(\bm 0),\partial_{\nu_{ij}}A(\bm 0)\rangle = -2\sigma_i\sigma_j$. Summarizing, in the basis \eqref{basis} the metric tensor $G$ has the following block-diagonal form
\vspace{0.3cm}

\begin{align}\label{G}
G=   \vphantom{% phantom stuff for correct box dimensions
    \begin{matrix}
    \overbrace{XYZ}^{\mbox{$ddd$}}\\ \\ \\ \\ \\ \\
    \underbrace{pqr}
    \end{matrix}}%
\begin{pmatrix}[ccc|ccc|ccc|ccc|ccc]
  \coolover{\substack{\mu_{ij},\\ 1\leq i<j\leq r}}{\phantom{x} & \phantom{\dots} & \phantom{x}} & \coolover{\substack{\nu_{k\ell},\\ 1\leq k<\ell\leq r}}{\phantom{x} & \phantom{\dots} & \phantom{x}} & \coolover{\substack{\mu_{ij},\\ 1\leq i\leq r<j\leq m}}{\phantom{x} & \phantom{\dots} & \phantom{x}} & \coolover{\substack{\nu_{k\ell},\\ 1\leq k\leq r<\ell\leq n}}{\phantom{x} & \phantom{\dots} & \phantom{x}} & \coolover{\substack{s_h,\\ 1\leq h\leq r}}{\phantom{x} & \phantom{\dots} & \phantom{x}} \\
  &  \text{\large $A$} &    &  &  \text{\large $B$} &  &  &  \mbox{\Large $0$} & &   &  \mbox{\Large $0$} &   & &  \mbox{\Large $0$} &  \\
  &   &    &  &   &  &  &   & &   &   &   & &   &\\ \hline
  &   &    &  &   &  &  &   & &   &   &   & &   &\\
  &  \text{\large $B$} &    &  &  \text{\large $A$} &  &  &  \mbox{\Large $0$} & &   &  \mbox{\Large $0$} &   & &  \mbox{\Large $0$} &  \\
   &   &    &  &   &  &  &   & &   &   &   & &   &\\ \hline
  &   &    &  &   &  &  &   & &   &   &   & &   &\\
  &  \mbox{\Large $0$} &    &  &  \mbox{\Large $0$} &  &  &  \text{\large $M$} & &   &  \mbox{\Large $0$} &   & &  \mbox{\Large $0$} &  \\
  &   &    &  &   &  &  &   & &   &   &   & &   &\\ \hline
    &   &    &  &   &  &  &   & &   &   &   & &   &\\
  &  \mbox{\Large $0$} &    &  &  \mbox{\Large $0$} &  &  &  \mbox{\Large $0$} & &   &  \text{\large $N$} &   & &  \mbox{\Large $0$} &  \\
  &   &    &  &   &  &  &   & &   &   &   & &   &\\ \hline
    &   &    &  &   &  &  &   & &   &   &   & &   &\\
  & \mbox{\Large $0$} &    &  &  \mbox{\Large $0$} &  &  &  \mbox{\Large $0$} & &   &  \mbox{\Large $0$} &   & & \text{\large{Id}} &  \\
   &   &    &  &   &  &  &   & &   &   &   & &   &\\ 
 \end{pmatrix}
  \begin{matrix}% matrix for right braces
  \hspace{-0.45cm}\coolrightbrace{\substack{\mu_{ij},\\ 1\leq i<j\leq r}}{x \\ \text{\Large A} \\ x} \\
  \hspace{-0.45cm}\coolrightbrace{\substack{\nu_{k\ell},\\ 1\leq k<\ell\leq r}}{x \\ \text{\Large A} \\ x} \\
   \coolrightbrace{\substack{\mu_{ij},\\ 1\leq i\leq r<j\leq m}}{x \\ \text{\Large A} \\ x} \\
   \coolrightbrace{\substack{\nu_{k\ell},\\ 1\leq k\leq r<\ell\leq n}}{x \\ \text{\Large A} \\ x} \\
  \hspace{-0.7cm}\coolrightbrace{\substack{s_h,\\ 1\leq h\leq r}}{x \\ \text{\Large A} \\ x} 
\end{matrix}\ ,
\end{align}
where the order in each of the indicated groups of rows and columns of $G$ is induced from the lexicographic order on sets of indices $(i,j)$, $(k,\ell)$ and    
\begin{align*}
\\    A&=
\begin{pmatrix}
\coolover{\substack{\mu_{ij},\\ 1\leq i<j\leq r}}{\ddots & \phantom{\sigma_i^2+\sigma_j^2} & \phantom{\ddots}} \\
    & \sigma_i^2+\sigma_j^2 & \\
    & & \ddots
  \end{pmatrix}
\coolrightbrace{\substack{\mu_{ij},\\ 1\leq i<j\leq r}}{\ddots \\ \sigma^2 \\ \ddots}
        =\begin{pmatrix}
    \coolover{\substack{\nu_{k\ell},\\ 1\leq k<\ell\leq r}}{\ddots & \phantom{\sigma_i^2+\sigma_j^2} & \phantom{\ddots}} \\
    & \sigma_k^2+\sigma_\ell^2 & \\
    & & \ddots
  \end{pmatrix}
        \coolrightbrace{\substack{\nu_{k\ell},\\ 1\leq k<\ell\leq r}}{\ddots \\ \sigma^2 \\ \ddots}\, ,\\
  \\
  B&= \begin{pmatrix}
\coolover{\substack{\nu_{k\ell},\\ 1\leq k<\ell \leq r}}{\ddots & \phantom{-2\sigma_{i=k}\sigma_{j=\ell}} & \phantom{\ddots}} \\
   & -2\sigma_{i=k}\sigma_{j=\ell} & \\
    & & \ddots
  \end{pmatrix}
        \coolrightbrace{\substack{\mu_{ij},\\ 1\leq i<j\leq r}}{\ddots \\ 2\sigma \\ \ddots}
  =
  \begin{pmatrix}
\coolover{\substack{\mu_{ij},\\ 1\leq i< j\leq r}}{\ddots & \phantom{-2\sigma_{i=k}\sigma_{j=\ell}} & \phantom{\ddots}} \\
   & -2\sigma_{k=i}\sigma_{\ell=j} & \\
    & & \ddots
  \end{pmatrix}
        \coolrightbrace{\substack{\nu_{k\ell},\\ 1\leq k<\ell\leq r}}{\ddots \\ 2\sigma_i \\ \ddots}
 \, ,\\
  \\
  M&=
  \begin{pmatrix}
\coolover{\substack{\mu_{ij},\\ 1\leq i\leq r<j\leq m}}{\ddots & \phantom{\sigma_i^2} & \phantom{\ddots}} \\
    & \sigma_{i}^2 & \\
    & & \ddots
  \end{pmatrix}
                \coolrightbrace{\substack{\mu_{ij},\\ 1\leq i\leq r<j\leq m}}{\ddots \\ \sigma^2 \\ \ddots}\, ,\\
  \\N&=
  \begin{pmatrix}
\coolover{\substack{\nu_{k\ell},\\ 1\leq k\leq r<\ell\leq n}}{\ddots & \phantom{\sigma_i^2} & \phantom{\ddots}}\\
    & \sigma_{k}^2 & \\
    & & \ddots
  \end{pmatrix}
                        \coolrightbrace{\substack{\nu_{k\ell},\\ 1\leq k\leq r<\ell\leq n}}{\ddots \\ \sigma_j^2 \\ \ddots}
\end{align*}
are diagonal square matrices of sizes ${r\choose 2}$, ${r\choose 2}$, $r(m-r)$, $r(n-r)$ respectively. Since $G$ is block-diagonal and blocks $A, B, M$ and $N$ are diagonal matrices, the inverse of $G$ equals
\vspace{0.3cm}
\begin{align}\label{Ginverse}
G^{-1}=   \vphantom{% phantom stuff for correct box dimensions
    \begin{matrix}
    \overbrace{XYZ}^{\mbox{$ddd$}}\\ \\ \\ \\ \\ \\
    \underbrace{pqr}
    \end{matrix}}%
\begin{pmatrix}[ccc|ccc|ccc|ccc|ccc]
  \coolover{\substack{\mu_{ij},\\ 1\leq i<j\leq r}}{ & \phantom{-\frac{A}{A^2-B^2}} & } & \coolover{\substack{\nu_{k\ell},\\ 1\leq k<\ell\leq r}}{ & \phantom{-\frac{\text{B}}{\text{A}^2-\text{B}^2}} & } & \coolover{\substack{\mu_{ij},\\ 1\leq i\leq r<j\leq m}}{\phantom{x} & \phantom{\text{M}^{-1}} & \phantom{x}} & \coolover{\substack{\nu_{k\ell},\\ 1\leq k\leq r<\ell\leq n}}{\phantom{x} & \phantom{N^{-1}} & \phantom{x}} & \coolover{\substack{s_h,\\ 1\leq h\leq r}}{\phantom{x} & \phantom{\text{\Large Id}} & \phantom{x}} \\
  &  \frac{A}{A^2-B^2} &    &  &  -\frac{B}{A^2-B^2} &  &  &  \mbox{\Large 0} & &   &  \mbox{\Large 0} &   & &  \mbox{\Large 0} &  \\
  &   &    &  &   &  &  &   & &   &   &   & &   &\\ \hline
  &   &    &  &   &  &  &   & &   &   &   & &   &\\
  &  -\frac{B}{A^2-B^2} &    &  &  \frac{A}{A^2-B^2} &  &  &  \mbox{\Large 0} & &   &  \mbox{\Large 0} &   & &  \mbox{\Large 0} &  \\
   &   &    &  &   &  &  &   & &   &   &   & &   &\\ \hline
  &   &    &  &   &  &  &   & &   &   &   & &   &\\
  &  \mbox{\Large 0} &    &  &  \mbox{\Large 0} &  &  &  M^{-1} & &   &  \mbox{\Large 0} &   & &  \mbox{\Large 0} &  \\
  &   &    &  &   &  &  &   & &   &   &   & &   &\\ \hline
    &   &    &  &   &  &  &   & &   &   &   & &   &\\
  &  \mbox{\Large 0} &    &  &  \mbox{\Large 0} &  &  &  \mbox{\Large 0} & &   &  N^{-1} &   & &  \mbox{\Large 0} &  \\
  &   &    &  &   &  &  &   & &   &   &   & &   &\\ \hline
    &   &    &  &   &  &  &   & &   &   &   & &   &\\
  & \mbox{\Large 0} &    &  &  \mbox{\Large 0} &  &  &  \mbox{\Large 0} & &   &  \mbox{\Large 0} &   & & \text{\large{Id}} &  \\
   &   &    &  &   &  &  &   & &   &   &   & &   &\\ 
 \end{pmatrix}
  \begin{matrix}% matrix for right braces
  \hspace{-0.45cm}\coolrightbrace{\substack{\mu_{ij},\\ 1\leq i<j\leq r}}{x \\ \text{\Large A} \\ x} \\
  \hspace{-0.45cm}\coolrightbrace{\substack{\nu_{k\ell},\\ 1\leq k<\ell\leq r}}{x \\ \text{\Large A} \\ x} \\
   \coolrightbrace{\substack{\mu_{ij},\\ 1\leq i\leq r<j\leq m}}{x \\ \text{\Large A} \\ x} \\
   \coolrightbrace{\substack{\nu_{k\ell},\\ 1\leq k\leq r<\ell\leq n}}{x \\ \text{\Large A} \\ x} \\
  \hspace{-0.7cm}\coolrightbrace{\substack{s_h,\\ 1\leq h\leq r}}{x \\ \text{\Large A} \\ x} 
\end{matrix}\ ,
\end{align}
with diagonal blocks
\begin{align*}
\\    \frac{A}{A^2-B^2}&=
\begin{pmatrix}
\coolover{\substack{\mu_{ij},\\ 1\leq i<j\leq r}}{\ddots & \phantom{\sigma_i^2+\sigma_j^2} & \phantom{\ddots}} \\
    & \frac{\sigma_i^2+\sigma_j^2}{(\sigma_i^2-\sigma_j^2)^2} & \\
    & & \ddots
  \end{pmatrix}
\coolrightbrace{\substack{\mu_{ij},\\ 1\leq i<j\leq r}}{\ddots \\ \frac{\sigma_i^2+\sigma_j^2}{(\sigma_i^2-\sigma_j^2)^2} \\ \ddots}
        =\begin{pmatrix}
    \coolover{\substack{\nu_{k\ell},\\ 1\leq k<\ell\leq r}}{\ddots & \phantom{\sigma_i^2+\sigma_j^2} & \phantom{\ddots}} \\
    & \frac{\sigma_k^2+\sigma_\ell^2}{(\sigma_k^2-\sigma_\ell^2)^2} & \\
    & & \ddots
  \end{pmatrix}
        \coolrightbrace{\substack{\nu_{k\ell},\\ 1\leq k<\ell\leq r}}{\ddots \\ \frac{\sigma_i^2+\sigma_j^2}{(\sigma_i^2-\sigma_j^2)^2} \\ \ddots}\, ,\\
  \\
  -\frac{B}{A^2-B^2}&= \begin{pmatrix}
\coolover{\substack{\nu_{k\ell},\\ 1\leq k<\ell \leq r}}{\ddots & \phantom{-2\sigma_{i=k}\sigma_{j=\ell}} & \phantom{\ddots}} \\
   & \frac{2\sigma_{i=k}\sigma_{j=\ell}}{(\sigma_{i=k}^2-\sigma_{j=\ell}^2)^2} & \\
    & & \ddots
  \end{pmatrix}
        \coolrightbrace{\substack{\mu_{ij},\\ 1\leq i<j\leq r}}{\ddots \\ \frac{2\sigma_{i=k}\sigma_{j=\ell}}{(\sigma_{i=k}^2-\sigma_{j=\ell}^2)^2} \\ \ddots}
  =
  \begin{pmatrix}
\coolover{\substack{\mu_{ij},\\ 1\leq i< j\leq r}}{\ddots & \phantom{-2\sigma_{i=k}\sigma_{j=\ell}} & \phantom{\ddots}} \\
   & \frac{2\sigma_{k=i}\sigma_{\ell=j}}{(\sigma_{k=i}^2-\sigma_{\ell=j}^2)^2} & \\
    & & \ddots
  \end{pmatrix}
        \coolrightbrace{\substack{\nu_{k\ell},\\ 1\leq k<\ell\leq r}}{\ddots \\ \frac{2\sigma_{i=k}\sigma_{j=\ell}}{(\sigma_{i=k}^2-\sigma_{j=\ell}^2)^2} \\ \ddots}
 \, ,\\
  \\
  M^{-1}&=
  \begin{pmatrix}
\coolover{\substack{\mu_{ij},\\ 1\leq i\leq r<j\leq m}}{\ddots & \phantom{\sigma_i^{-2}} & \phantom{\ddots}} \\
    & \sigma_{i}^{-2} & \\
    & & \ddots
  \end{pmatrix}
                \coolrightbrace{\substack{\mu_{ij},\\ 1\leq i\leq r<j\leq m}}{\ddots \\ \sigma^2 \\ \ddots}\, ,\\
  \\N^{-1}&=
  \begin{pmatrix}
\coolover{\substack{\nu_{k\ell},\\ 1\leq k\leq r<\ell\leq n}}{\ddots & \phantom{\sigma_i^{=2}} & \phantom{\ddots}}\\
    & \sigma_{k}^{-2} & \\
    & & \ddots
  \end{pmatrix}
                        \coolrightbrace{\substack{\nu_{k\ell},\\ 1\leq k\leq r<\ell\leq n}}{\ddots \\ \sigma_j^2 \\ \ddots}\, .
\end{align*}
Now, \eqref{mean2} implies that the mean curvature vector of $\M_{m,n,r}$ at $A$ is equal to
\begin{align}\label{H_A}
  H_A=\textrm{\normalfont{Tr}}\left[G^{-1} \left(\mathrm{d}^2A(\bm 0)\right)^\perp\right],
\end{align}
where $\mathrm{d}^2 A(\bm 0)$ is the matrix of second partial derivatives of \eqref{par} computed at $(\mu,\nu,s)=(\bm 0, \bm 0,\bm 0)$ and $(\mathrm{d}^2A(\bm 0))^\perp$ stands for its normal component (applied entry-wise to $\mathrm{d}^2A(\bm 0)$). From \eqref{mean} we see that in order to compute $H_A$ we need only those second partial derivatives of \eqref{par} that correspond to non-zero entries of $G^{-1}$. For $1\leq i<j\leq m$, $i\leq r$, for $1\leq h\leq r$ and for $1\leq k<\ell\leq n$, $k\leq r$, we have, using \eqref{der_mu}, \eqref{der_s} and \eqref{der_nu}, 
\small\begin{equation}\label{partial_1}
\begin{aligned}
 \partial_{\mu_{ij}}\partial_{\mu_{ij}} A(\bm 0) &= L^2_{ij}\left(\sum_{h=1}^r \sigma_hE_{hh}\right) =(-E_{ii}-E_{jj})\left(\sum_{h=1}^r\sigma_hE_{hh}\right)= -\sigma_iE_{ii}-\sigma_jE_{jj},\\
  \partial_{s_h}\partial_{s_h}A(\bm 0) &= 0,\\
  \partial_{\nu_{k\ell}}\partial_{\nu_{k\ell}} A(\bm 0) &= \left(\sum_{h=1}^r \sigma_hE_{hh}\right)L^2_{k\ell} = \left(\sum_{h=1}^r\sigma_hE_{hh}\right)(-E_{kk}-E_{\ell\ell})=-\sigma_kE_{kk}-\sigma_\ell E_{\ell\ell},
\end{aligned}
\end{equation}
where we recall that $\sigma_j=\sigma_\ell=0$ for $j,\ell>r$. Finally, for $1\leq i=k<j=\ell\leq r$ we have
\begin{align}\label{partial_2}
    \partial_{\mu_{ij}}\partial_{\nu_{ij}}A(\bm 0) = L_{ij}\left(\sum_{h=1}^r\sigma_hE_{hh}\right)(-L_{ij}) = \sigma_{j}E_{ii}+\sigma_iE_{jj}. 
\end{align}
From the form \eqref{Ginverse} of $G^{-1}$ we see that \eqref{partial_1} and \eqref{partial_2} are the only second partial derivatives of \eqref{par} that matter for the formula \eqref{H_A} of $H_A$. Before we proceed to computing normal components of these matrices let us describe the normal space to $\M_{m,n,r}$ at $A=A(\bm 0)$. Matrices
\begin{align}\label{normal}
  E_{pq},\quad r<p\leq m,\ r<q\leq n,
\end{align}
are orthogonal to matrices \eqref{basis} that form a basis of the tangent space $T_A\M_{m,n,r}$. Moreover, \eqref{normal} are independent and there are $(m-r)(n-r)=mn-((m+n)r-r^2) = \textrm{\normalfont{codim}}(\M_{m,n,r})$ many of them. It implies that matrices \eqref{normal} form a basis of the normal space to $\M_{m,n,r}$ at $A$. It is now straightforward to see that \eqref{partial_1} and \eqref{partial_2} are orthogonal to \eqref{normal} or, equivalently, they have trivial normal components. It follows then from the above reasoning that $H_A=0$.

In the beginning of the proof we assumed that the non-zero singular values of $A\in \M_{m,n,r}$ are distinct. Those $A\in \M_{m,n,r}$ that do not satisfy this assumption form an algebraic submanifold $X\subset \M_{m,n,r}$ which is proper because there are obviously matrices in $\M_{m,n,r}\setminus X$. Since the mean curvature vector field $H$ is a smooth field of normal vectors to $\M_{m,n,r}$ and since $H_A=0$ for $A$ in the open and dense subset $\M_{m,n,r}\setminus X\subset \M_{m,n,r}$, we have $H_A=0$ for all $A\in \M_{m,n,r}$.
\qed
\subsection{Proof of Theorem \ref{thm_skew}}\label{sub_skew}
We pursue the same strategy as in the proof of Theorem \ref{thm_general}. Let us write a skew-symmetric matrix $A\in \Sk_{n,r}$ of rank $2r<n$ in the normal form \eqref{normalform}, $A=V^\mathsf{T}\Omega V$, where $V\in O(n)$ and
\begin{align}\label{omega2}
  \Omega = \sum_{h=1}^r \omega_hL_{2h,2h-1}
\end{align}
is as in \eqref{omega}, where recall $L_{ij} = E_{ji}-E_{ij}$. We first assume that $\omega_1>\dots>\omega_r>0$. Since the Frobenius inner product \eqref{Frob} is invariant under the action \eqref{action} of $O(n)\times O(n)$ on $\M_{n,n}$ it is, in particular, invariant under the diagonal subaction \eqref{skaction} of $O(n)$ on $\Sk_{n}$. Thus, in order to prove that $\Sk_{n,2r}$ is a minimal submanifold of $(\Sk_n,\langle \cdot,\cdot\rangle)$ it is enough to show that for any block-diagonal matrix $A=\Omega\in \Sk_{n,2r}$ \eqref{omega2} the mean curvature vector \eqref{mean2} at $A=\Omega$ is zero. For this let us consider the following parametrization of a neighborhood of $A=\Omega$:
\begin{equation}\label{skpar}
\small\begin{aligned}
  A(\bm \mu, \bm s) = \left(\prod_{\substack{1\leq i<j\leq n,\, i\leq 2r,\\ (i,j)\neq (2t-1,2t)}} e^{\mu_{ij}L_{ij}}\right)^\mathsf{T}\left(\sum_{h=1}^r (\omega_h+s_h)L_{2h,2h-1}\right)\left(\prod_{\substack{1\leq k<\ell\leq n,\, k\leq 2r,\\ (k,\ell)\neq (2t-1,2t)}} e^{\mu_{k\ell}L_{k\ell}}\right), 
\end{aligned}
\end{equation}
where orthogonal matrices $e^{\mu_{ij}L_{ij}}$ in the product in \eqref{skpar} are ordered according to the lexicographic order on the set of indices $(i,j)$ and $(i,j)\neq (2t-1,2t)$ means that $i<j$ are not consecutive integers with $i$ being odd. Exactly in the same way as in the proof of Theorem \ref{thm_general} we compute first order derivatives of the parametrization \eqref{skpar}. We have
\begin{equation}\label{skder_mu}
\footnotesize\begin{aligned}
  \partial_{\mu_{ij}}A(\bm \mu,\bm s) &= \left(e^{\mu_{13}L_{13}}\dots L_{ij}e^{\mu_{ij}L_{ij}}\dots e^{\mu_{2r,n}L_{2r, n}}\right)^\mathsf{T}\left(\sum_{h=1}^r(\omega_h+s_h)L_{2h,2h-1}\right)\left(\prod_{\substack{1\leq k<\ell\leq n,\, k\leq 2r,\\ (k,\ell)\neq (2t-1,2t)}} e^{\mu_{k\ell}L_{k\ell}}\right)\\
&+\left(\prod_{\substack{1\leq k<\ell\leq n,\, k\leq 2r,\\ (k,\ell)\neq (2t-1,2t)}} e^{\mu_{k\ell}L_{k\ell}}\right)^\mathsf{T}\left(\sum_{h=1}^r(\omega_h+s_h)L_{2h,2h-1}\right)\left(e^{\mu_{13}L_{13}}\dots L_{ij}e^{\mu_{ij}L_{ij}}\dots e^{\mu_{2r,n}L_{2r, n}}\right)
\end{aligned}
\end{equation}
for $1\leq i<j\leq n$, $i\leq 2r$ with $(i,j)\neq (2t-1,2t)$ and
\begin{equation}\label{skder_s}
  \begin{aligned}
    \partial_{s_h}A(\bm \mu, \bm s) = \left(\prod_{\substack{1\leq i<j\leq n,\, i\leq 2r,\\ (i,j)\neq (2t-1,2t)}} e^{\mu_{ij}L_{ij}}\right)^\mathsf{T} L_{2h,2h-1}   \left(\prod_{\substack{1\leq k<\ell\leq n,\, k\leq 2r,\\ (k,\ell)\neq (2t-1,2t)}} e^{\mu_{k\ell}L_{k\ell}}\right)
  \end{aligned}
\end{equation}
for $1\leq h\leq r$. Matrices \eqref{skder_mu}, \eqref{skder_s} belong to the tangent space to $\Sk_{n,2r}$ at $A(\bm \mu,\bm s)$. At $A(\bm 0):=A(\bm 0,\bm 0) = A=\Omega$ these are equal
\begin{equation}\label{mu=0}
\begin{aligned}
  \partial_{\mu_{ij}}A(\bm 0) &= -L_{ij}\Omega+\Omega L_{ij} = \sum_{h=1}^r \omega_h\, (L_{2h,2h-1}L_{ij}-L_{ij}L_{2h,2h-1})\\
                              &=\sum_{h=1}^r\omega_h \left( L_{i,2h-1} \delta_{j,2h}-L_{j,2h-1}\delta_{i,2h}-L_{i,2h}\delta_{j,2h-1}+L_{j,2h}\delta_{i,2h-1}\right)
\end{aligned}
\end{equation}
for $1\leq i<j\leq n$, $i\leq 2r$, $(i,j)\neq (2t-1,2t)$, where $\delta_{ij}$ is the Kronecker delta symbol, and
\begin{align}\label{s=0}
\partial_{s_h}A(\bm 0) = L_{2h,2h-1},\quad 1\leq h\leq r.
\end{align}
We now elaborate \eqref{mu=0} further and distinguish the following four cases depending on the parity of the indices $i$ and $j$:
\begin{enumerate}[label=(\roman*)]
\item $1\leq i=2p<j=2q\leq n$, $p\leq r$, (both $i$ and $j$ are even)
  \begin{align}\label{elab1}
    \partial_{\mu_{ij}}A(\bm 0) = \omega_p L_{2p-1,2q}+\omega_q L_{2p,2q-1}
  \end{align}
\item $1\leq i=2p<j=2q-1\leq n$, $p\leq r$, ($i$ is even and $j$ is odd)
  \begin{align}\label{elab2}
    \partial_{\mu_{ij}}A(\bm 0) = \omega_p L_{2p-1,2q-1}-\omega_q L_{2p,2q}
  \end{align}
\item $1\leq i=2p-1<j=2q\leq n$, $p\leq r$, $p<q$, ($i$ is odd and $j$ is even)
  \begin{align}\label{elab3}
    \partial_{\mu_{ij}}A(\bm 0) = -\omega_p L_{2p,2q}+\omega_q L_{2p-1,2q-1}
  \end{align}
\item $1\leq i=2p-1<j=2q-1\leq n$, $p\leq r$, (both $i$ and $j$ are odd)
  \begin{align}\label{elab4}
    \partial_{\mu_{ij}}A(\bm 0) = -\omega_p L_{2p,2q-1}-\omega_q L_{2p-1,2q}
  \end{align}
\end{enumerate}
where we set $\omega_q=0$ whenever $r<q$. If $2r<j$ (equivalently, $r<q$), then one can write
\begin{equation}\label{r<q}
\begin{aligned}
  \partial_{\mu_{ij}}A(\bm 0) &= \omega_{\floor{\frac{i+1}{2}}} L_{i-1,j}\quad \textrm{for even $i$ and}\\
  \partial_{\mu_{ij}}A(\bm 0) &= -\omega_{\floor{\frac{i+1}{2}}} L_{i+1,j}\quad \textrm{for odd $i$}.
\end{aligned}
\end{equation}
There are
\begin{align}\label{skcount}
  {n \choose 2} - {n-2r \choose 2} - r + r = {n\choose 2} - {n-2r\choose 2}
\end{align}
matrices in \eqref{mu=0} and \eqref{s=0} and one can see that they are linearly independent using \eqref{elab1}, \eqref{elab2}, \eqref{elab3}, \eqref{elab4}. The count \eqref{skcount} and the formula $\dim(\Sk_{n,2r}) = {n\choose 2}-{n-2r\choose 2}$ \cite[p.$49$]{FL1983} imply that \eqref{skpar} is indeed a parametrization of $\Sk_{n,2r}$ around $A(\bm 0)=A$ and, in particular, matrices \eqref{mu=0} and \eqref{s=0} form a basis of the tangent space to $\Sk_{n,2r}$ at $A=\Omega$. We want to compute the metric tensor of $\Sk_{n,2r}$ at $A$ in this basis. For this let us first fix a particular order of the tangent vectors \eqref{mu=0} and \eqref{s=0}. For each pair of indices $(p,q)$ such that $1\leq p<q\leq r$ we form the group  $\{\partial_{\mu_{2p-1,2q-1}}A(\bm 0)$, $\partial_{\mu_{2p-1,2q}}A(\bm 0)$, $\partial_{\mu_{2p,2q-1}}A(\bm 0)$, $\partial_{\mu_{2p,2q}}A(\bm 0)\}$ of four tangent vectors. We then order these \emph{$(p,q)$-groups} according to the lexicographic order on the set of indices $(p,q)$, $1\leq p<q\leq r$. Next, we put lexicographically ordered vectors $\partial_{\mu_{ij}}A(\bm 0)$, $1\leq i\leq 2r<j\leq n$. Finally, we put $\partial_{s_h}A(\bm 0)$, $1\leq h\leq r$. From \eqref{mu=0} and \eqref{s=0} we see that vectors $\partial_{s_h} A(\bm 0)$ are orthogonal to all $\partial_{\mu_{ij}}A(\bm 0)$ and they are also orthogonal among themselves with $\langle \partial_{s_h}A(\bm 0), \partial_{s_h}A(\bm 0)\rangle = 2$. Each $\partial_{\mu_{ij}}A(\bm 0)$, $1\leq i\leq 2r<j\leq n$, is orthogonal to all other vectors and $\langle \partial_{\mu_{ij}}A(\bm 0),\partial_{\mu_{ij}}A(\bm 0)\rangle = 2\,\omega_{\floor{\frac{i+1}{2}}}^2$. Vectors from two different $(p,q)$-groups are orthogonal to each other and the Gram matrix of a given $(p,q)$-group equals
\vspace{0.3cm}\begin{align}\label{pq}
2G_{p,q}= 2 \vphantom{% phantom stuff for correct box dimensions
    \begin{matrix}
    \overbrace{XYZ}^{\mbox{$ddd$}} \\ \\ \\  
    \underbrace{pqr}
    \end{matrix}}
\begin{pmatrix}
  \coolover{{\scriptstyle(2p-1,2q-1)}}{\phantom{\ }\omega_p^2+\omega_q^2} & \coolover{{\scriptstyle (2p-1,2q)}}{\phantom{aaa}0\phantom{xaa}} & \coolover{{\scriptstyle (2p,2q-1)}}{\phantom{aaa}0\phantom{xaa}} & \coolover{{\scriptstyle (2p,2q)}}{-2\,\omega_p\,\omega_q} \\ 
   0  & \omega_p^2+\omega_q^2  & 2\,\omega_p\,\omega_q  &  0  \\ 
  0 &  2\,\omega_p\,\omega_q &  \omega_p^2+\omega_q^2   & 0  \\ 
  -2\,\omega_p\,\omega_q  & 0 & 0 &  \omega_p^2+\omega_q^2
\end{pmatrix}
  \begin{matrix}% matrix for right braces
  \hspace{-0.15cm}\coolrightbrace{{\scriptstyle(2p-1,2q-1)}}{ \omega^2 } \\
 \hspace{-0.5cm} \coolrightbrace{{\scriptstyle(2p-1,2q)}}{ \omega^2 } \\
 \hspace{-0.5cm}  \coolrightbrace{{\scriptstyle(2p,2q-1)}}{ \omega^2} \\
  \hspace{-0.9cm} \coolrightbrace{{\scriptstyle(2p,2q)}}{ \omega^2} 
\end{matrix}.
              \end{align}
             Summarizing, the metric tensor $G$ in the basis \eqref{mu=0}, \eqref{s=0} equals

              \begin{align}\label{skG}
G= 2 \vphantom{% phantom stuff for correct box dimensions
\begin{matrix}
  \overbrace{XYZ}^{\mbox{$ddd$}} \\ \\  \\ \\ \\  \\ \\ \\ 
    \underbrace{pqr}
    \end{matrix}}
\begin{pmatrix}[ccc|ccc|ccc]
  \coolover{{\scriptstyle 1\leq p<q\leq r}}{\phantom{\ }\ddots\phantom{\ }  & \phantom{G_{p,q}} & \phantom{\ }\mbox{\large 0}\phantom{\ }} & \coolover{{\scriptstyle 1\leq i\leq 2r<j\leq n}}{\vspace{-0.1cm}\phantom{\ddots} & \phantom{\omega_{\floor{\frac{i+1}{2}}}} & \phantom{\ddots}} & \coolover{{\scriptstyle 1\leq h\leq r}}{\phantom{\ddots} & \phantom{\text{\large{Id}}} & \phantom{\ddots}} \\
  & G_{p,q} & \phantom{\mbox{\huge 0}}  &  & \mbox{\huge 0} &    &          & \mbox{\huge 0} &  \\ 
    \mbox{\large 0} & & \ddots & & & & & & \\ \hline
    & &   &   \ddots & & \mbox{\large 0} &    & & \\
    & \mbox{\huge 0} &   &    & \omega_{\floor{\frac{i+1}{2}}}^2 &   &    & \mbox{\huge 0}& \\
    & &   &  \mbox{\large 0}  &                             &\ddots   &    & & \\ \hline
    & &   &    &                     &                 &    & & \\
    &\mbox{\huge 0} &  &    &\mbox{\huge 0} &    &    & \mbox{\large Id} & \\
    & & & & & & & & 
  \end{pmatrix}
\begin{matrix}% matrix for right braces
\vspace{-0.45cm}\\\hspace{-0.5cm}\coolrightbrace{{\scriptstyle 1\leq p<q\leq r}}{ \left(\frac{n}{2}\right)\ddots\mbox{\large 0}  \\  \mbox{\huge 0} \\\mbox{\large 0}\ddots } \\
 \hspace{0cm} \coolrightbrace{{\scriptstyle 1\leq i\leq 2r<j\leq n}}{\mbox{\large 0}\ddots \\  \mbox{\huge 0} \omega_{\floor{\frac{i+1}{2}}}^2\\ \ddots\mbox{\large 0} } \\
 \hspace{-0.8cm}  \coolrightbrace{{\scriptstyle 1\leq h\leq r}}{xx \\  \mbox{\huge 0} xx\\ xx } 
\end{matrix}           \,   ,   
              \end{align}
              where the upper-left corner is block-diagonal with $4\times 4$ blocks $G_{p,q}$ from \eqref{pq}. Since $G$ has block-diagonal form its inverse $G^{\, -1}$ equals

              \vspace{0.4cm}                           \begin{align}\label{skGinverse}
G^{\,-1}= \frac{1}{2} \vphantom{% phantom stuff for correct box dimensions
\begin{matrix}
  \overbrace{XYZ}^{\mbox{$ddd$}} \\ \\  \\ \\ \\  \\ \\ \\ 
    \underbrace{pqr}
    \end{matrix}}
\begin{pmatrix}[ccc|ccc|ccc]
  \coolover{{\scriptstyle 1\leq p<q\leq r}}{\phantom{\ }\ddots\phantom{\ }  & \phantom{G_{p,q}} & \phantom{\ }\mbox{\large 0}\phantom{\ }} & \coolover{{\scriptstyle 1\leq i\leq 2r<j\leq n}}{\vspace{-0.1cm}\phantom{\ddots} & \phantom{\omega_{\floor{\frac{i+1}{2}}}} & \phantom{\ddots}} & \coolover{{\scriptstyle 1\leq h\leq r}}{\phantom{\ddots} & \phantom{\text{\large{Id}}} & \phantom{\ddots}} \\
  & G_{p,q}^{-1} & \phantom{\mbox{\huge 0}}  &  & \mbox{\huge 0} &    &          & \mbox{\huge 0} &  \\ 
    \mbox{\large 0} & & \ddots & & & & & & \\ \hline
    & &   &   \ddots & & \mbox{\large 0} &    & & \\
    & \mbox{\huge 0} &   &    & \omega_{\floor{\frac{i+1}{2}}}^{-2} &   &    & \mbox{\huge 0}& \\
    & &   &  \mbox{\large 0}  &                             &\ddots   &    & & \\ \hline
    & &   &    &                     &                 &    & & \\
    &\mbox{\huge 0} &  &    &\mbox{\huge 0} &    &    & \mbox{\large Id} & \\
    & & & & & & & & 
  \end{pmatrix}
\begin{matrix}% matrix for right braces
\vspace{-0.45cm}\\\hspace{-0.5cm}\coolrightbrace{{\scriptstyle 1\leq p<q\leq r}}{ \left(\frac{n}{2}\right)\ddots\mbox{\large 0}  \\  \mbox{\huge 0} \\\mbox{\large 0}\ddots } \\
 \hspace{0cm} \coolrightbrace{{\scriptstyle 1\leq i\leq 2r<j\leq n}}{\mbox{\large 0}\ddots \\  \mbox{\huge 0} \omega_{\floor{\frac{i+1}{2}}}^2\\ \ddots\mbox{\large 0} } \\
 \hspace{-0.8cm}  \coolrightbrace{{\scriptstyle 1\leq h\leq r}}{xx \\  \mbox{\huge 0} xx\\ xx } 
\end{matrix}           \,   ,   
              \end{align} 
              where inverses of blocks $G_{p,q}$ are easily found to be

              \vspace{0.3cm}\begin{align*}
G_{p,q}^{-1}= \frac{1}{(\omega_p^2-\omega_q^2)^2}\vphantom{% phantom stuff for correct box dimensions
    \begin{matrix}
    \overbrace{XYZ}^{\mbox{$ddd$}} \\ \\ \\  
    \underbrace{pqr}
    \end{matrix}}
\begin{pmatrix}
  \coolover{{\scriptstyle(2p-1,2q-1)}}{\phantom{\ }\omega_p^2+\omega_q^2} & \coolover{{\scriptstyle (2p-1,2q)}}{\phantom{aaa}0\phantom{xaa}} & \coolover{{\scriptstyle (2p,2q-1)}}{\phantom{aaa}0\phantom{xaa}} & \coolover{{\scriptstyle (2p,2q)}}{2\,\omega_p\,\omega_q} \\ 
   0  & \omega_p^2+\omega_q^2  & -2\,\omega_p\,\omega_q  &  0  \\ 
  0 & - 2\,\omega_p\,\omega_q &  \omega_p^2+\omega_q^2   & 0  \\ 
  2\,\omega_p\,\omega_q  & 0 & 0 &  \omega_p^2+\omega_q^2
\end{pmatrix}
  \begin{matrix}% matrix for right braces
  \hspace{-0.15cm}\coolrightbrace{{\scriptstyle(2p-1,2q-1)}}{ \omega^2 } \\
 \hspace{-0.5cm} \coolrightbrace{{\scriptstyle(2p-1,2q)}}{ \omega^2 } \\
 \hspace{-0.5cm}  \coolrightbrace{{\scriptstyle(2p,2q-1)}}{ \omega^2} \\
  \hspace{-0.9cm} \coolrightbrace{{\scriptstyle(2p,2q)}}{ \omega^2} 
\end{matrix}.
              \end{align*}              
              Recall from \eqref{mean2} that the mean curvature of $\Sk_{n,2r}$ at $A$ can be computed as
              \begin{align}\label{skH_A}
  H_A= \textrm{\normalfont{Tr}}\left[G^{-1} \left(\mathrm{d}^2A(\bm 0)\right)^\perp \right],
\end{align}
where $\mathrm{d}^2 A(\bm 0)$ is the matrix of second partial derivatives of \eqref{skpar} computed at $(\mu,s)=(\bm 0, \bm 0)$ and $(\mathrm{d}^2A(\bm 0))^\perp$ stands for its normal component (applied entry-wise to $\mathrm{d}^2A(\bm 0)$). As in the proof of Theorem \ref{thm_general} we see from \eqref{mean} that in order to compute $H_A$ we need only those second partial derivatives of \eqref{skpar} that correspond to non-zero entries of $G^{-1}$. So, the derivatives we have to look at are $\partial_{s_h}\partial_{s_h}A(\bm 0)$, where $1\leq h\leq r$, $\partial_{\mu_{ij}}\partial_{\mu_{ij}} A(\bm 0)$, where $1\leq i\leq 2r<j\leq n$, and $\partial_{\mu_{ij}}\partial_{\mu_{k\ell}} A(\bm 0)$, where $(i,j)$ and $(k,\ell)$ belong to the same $(p,q)$-block. First of all, $\partial_{s_h}\partial_{s_h} A(\bm 0) = 0$ for $1\leq h\leq r$. Second, from \eqref{skder_mu} and \eqref{mu=0} we have that for any $1\leq i\leq 2r<j\leq n$ 
\begin{equation}\label{ijij}
\begin{aligned}
  \partial_{\mu_{ij}}\partial_{\mu_{ij}}A(\bm 0) &= L_{ij}^2\Omega - 2L_{ij}\Omega L_{ij} + \Omega L_{ij}^2=-L_{ij}\,\partial_{\mu_{ij}}A(\bm 0)+\partial_{\mu_{ij}}A(\bm 0) L_{ij}\\
  &=
     \omega_{\floor{\frac{i+1}{2}}}\cdot  \begin{cases}
       &-L_{ij}L_{i-1,j}+L_{i-1,j}L_{ij}\quad \textrm{if $i$ is even}\\ 
       &L_{ij}L_{i+1,j} - L_{i+1,j}L_{ij}\quad \textrm{if $i$ is odd}
     \end{cases}\\
  & =\omega_{\floor{\frac{i+1}{2}}}\cdot\begin{cases}
       & L_{i-1,i}\quad \textrm{if $i$ is even}\\ 
       & L_{i,i+1}\quad \textrm{if $i$ is odd}
     \end{cases},
   \end{aligned}
   \end{equation}
where in the second equality we used \eqref{r<q}.
Third, for any $(i,j)$ and $(k,\ell)$ with $(i,j)\leq (k,\ell)$ (in the lexicographic order) \eqref{skder_mu} and \eqref{mu=0} imply that
\small\begin{equation}\label{ijkl}
  \partial_{\mu_{ij}}\partial_{\mu_{kl}}A(\bm 0) = L_{k\ell}L_{ij}\Omega - L_{k\ell}\Omega L_{ij}-L_{ij}\Omega L_{k\ell} + \Omega L_{ij}L_{k\ell}=-L_{k\ell}\,\partial_{\mu_{ij}}A(\bm 0)+\partial_{\mu_{ij}}A(\bm 0) L_{k\ell}.
\end{equation}
Assuming that the indices $(i,j)$ and $(k,\ell)$ belong to the same $(p,q)$-group and using \eqref{elab1}, \eqref{elab2}, \eqref{elab3} and \eqref{elab4} we now elaborate \eqref{ijkl}, conveniently storing the result in a matrix, 

\vspace{0.3cm}  \begin{align}\label{4x4}
    \left(\partial_{\mu_{ij}}\partial_{\mu_{k\ell}} A(\bm 0)\right)=
    \vphantom{% phantom stuff for correct box dimensions
    \begin{matrix}
    \overbrace{XYZ}^{\mbox{$ddd$}} \\ \\ \\  
    \underbrace{pqr}
    \end{matrix}}
\begin{pmatrix}
  \coolover{{\scriptstyle(2p-1,2q-1)}}{\substack{\omega_pL_{2p-1,2p}\\+\omega_qL_{2q-1,2q}}} & \coolover{{\scriptstyle (2p-1,2q)}}{\vspace{0.1cm}\phantom{\,xxx}0\phantom{\,xxx}} & \coolover{{\scriptstyle (2p,2q-1)}}{\vspace{0.1cm}\phantom{\,xxx}0\phantom{\,xxx}} & \coolover{{\scriptstyle (2p,2q)}}{\substack{-\omega_pL_{2q-1,2q}\\-\omega_qL_{2p-1,2p}}} \\ 
   0  & \substack{\omega_pL_{2p-1,2p}\\+\omega_qL_{2q-1,2q}}  & \substack{\omega_pL_{2q-1,2q}\\+\omega_qL_{2p-1,2p}} &  0  \vspace{0.2cm}\\
  0 &  \substack{\omega_pL_{2q-1,2q}\\+\omega_qL_{2p-1,2p}} &  \substack{\omega_pL_{2p-1,2p}\\+\omega_qL_{2q-1,2q}}   & 0  \\ 
  \substack{-\omega_pL_{2q-1,2q}\\-\omega_qL_{2p-1,2p}}  & 0 & 0 &  \substack{\omega_pL_{2p-1,2p}\\+\omega_qL_{2q-1,2q}}
\end{pmatrix}
  \begin{matrix}% matrix for right braces
  \hspace{-0.15cm}\coolrightbrace{{\scriptstyle(2p-1,2q-1)}}{ \substack{\omega_pL_{2p-1,2p}\\+\omega_qL_{2q-1,2q}} } \\
 \hspace{-0.5cm} \coolrightbrace{{\scriptstyle(2p-1,2q)}}{ \substack{\omega_pL_{2p-1,2p}\\+\omega_qL_{2q-1,2q}}} \\
 \hspace{-0.5cm}  \coolrightbrace{{\scriptstyle(2p,2q-1)}}{ \substack{\omega_pL_{2p-1,2p}\\+\omega_qL_{2q-1,2q}}} \\
  \hspace{-0.9cm} \coolrightbrace{{\scriptstyle(2p,2q)}}{\substack{\omega_pL_{2p-1,2p}\\+\omega_qL_{2q-1,2q}}} 
\end{matrix}.
              \end{align}
            In order to find normal components of matrices \eqref{ijij} and \eqref{ijkl} we need to describe the normal space to $\Sk_{n,2r}$ at $A$. For this let us observe that matrices
            \begin{align}\label{sknormal}
              L_{ab},\quad 2r<a<b\leq n,
            \end{align}
            are orthogonal to matrices \eqref{mu=0}, \eqref{s=0} that form a basis of the tangent space $T_A\Sk_{n,2r}$. Moreover,  \eqref{sknormal} are independent and there are ${n-2r \choose 2} = \textrm{\normalfont{codim}}(\Sk_{n,2r})$ many of them. It implies that matrices \eqref{sknormal} form a basis of the normal space to $\Sk_{n,2r}$ at $A$. It is now elementary to check that the relevant second derivatives \eqref{ijij} and \eqref{4x4} are orthogonal to \eqref{sknormal} or, equivalently, they have trivial normal components. It follows from the above reasoning that $H_A=0$.
            
In the beginning of the proof we assumed that the non-zero eigenvalues of $A\in \Sk_{n,2r}$ are distinct. Those $A\in \Sk_{n,2r}$ that do not satisfy this assumption form an algebraic submanifold $X\subset \Sk_{n,2r}$ which is proper because there are obviously matrices in $\Sk_{n,2r}\setminus X$. Since the mean curvature vector field $H$ is a smooth field of normal vectors to $\Sk_{n,2r}$ and since $H_A=0$ for $A$ in the open and dense subset $\Sk_{n,2r}\setminus X\subset \Sk_{n,2r}$, we have $H_A=0$ for all $A\in \Sk_{n,2r}$.\qed 

\subsection{Proof of Theorem \ref{thm_sym}}\label{sub_sym}

Our proof uses a general result of Hsiang and Lawson from \cite{HL1971} which we state now.
Let $(M,\bm g)$ be a Riemannian manifold and let $G$ be a compact, connected group acting smoothly by isometries on $M$. For $x\in M$ let $G_x$ be the stabilizer of $x$, that is, the group of those transformations in $G$ that fix $x$. Let us consider the equivalence relation on $M$ for which points $x, y\in M$ are said to be equivalent if their stabilizers $G_x$ and $G_y$ are conjugate, that is, $g\,G_xg^{-1}=G_y$ for some $g\in G$. By \cite[Sec. $1.3$]{HL1971}, equivalence classes of this relation are minimal submanifolds of $(M,\bm g)$.

Let us now apply this result to $(M,\bm g) = (\Sym_n,\langle \cdot,\cdot\rangle)$ and the action \eqref{symaction} of $G=O(n)$. Recall from Subsection \ref{svd} that this action preserves the inner product \eqref{Frob}.   
By Corollary \ref{conj}, sets $\Sym_{n,\vec\kappa}\subset \Sym_n$ are precisely the equivalence classes of the above defined equivalence relation on $\Sym_n$ and hence they are minimal submanifolds of $(\Sym_n,\langle\cdot,\cdot\rangle)$. \qed  

\bibliographystyle{amsalpha}

\end{document}